\documentclass{article}

\usepackage{amssymb}
\usepackage{amsmath}

\usepackage{diagrams}
\newtheorem{theorem}{Theorem}[section]

\newtheorem{definition}[theorem]{Definition}
\newtheorem{example}[theorem]{Example}

\newtheorem{proposition}[theorem]{Proposition}
\newtheorem{remark}[theorem]{Remark}

\newenvironment{proof}{\paragraph{\normalfont \textit{Proof:}}}{\hfill$\square$}

\newarrow {Partial} ----{harpoon}
\newarrow{Into} C--->
\newarrow{Onto} ----{>>}

\begin{document}

\title{A classifying groupoid for compact Hausdorff locales}

\author{Simon Henry, Christopher Townsend}
% \date{May 2023}
\maketitle

\begin{abstract}
We construct a localic groupoid $\mathbb{G}_{KH}$ such that for any locale $X$ the category of compact Hausdorff locales in the topos of sheaves over $X$ is equivalent to a category whose objects are principal $\mathbb{G}_{KH}$-bundles over $X$ and whose morphisms are $\mathbb{S}$-homotopies (where $\mathbb{S}$ is the Sierpi\'{n}ski locale).

This result can be intuitively viewed as the compact Hausdorff dual of the well known result from topos theory that there is an object classifier. 
\end{abstract}

\section{Introduction}
The paper \cite{BDesc} proves that for any \'{e}tale-complete localic groupoid $\mathbb{G}$, if we consider the topos $B(\mathbb{G})$ of $\mathbb{G}$-equivariant sheaves, then  geometric morphisms $Sh(X) \rTo B(\mathbb{G})$ are in bijection with principal $\mathbb{G}$-bundles over $X$. By the famous Joyal and Tierney result (\cite{JoyT}) we know that every bounded topos is of the form $B(\mathbb{G})$ for some \'{e}tale-complete localic groupoid $\mathbb{G}$ and so since there is an object classifier this means we can find a localic groupoid $\mathbb{G}$ and a bijection between principal $\mathbb{G}$-bundles over $X$ and $Sh(X)$ for any locale $X$. Recalling that $Sh(X)$ can be identified with discrete locales internal to $Sh(X)$ we can therefore identify discrete locales with principal $\mathbb{G}$-bundles for some localic groupoid $\mathbb{G}$. The purpose of this paper is to prove a compact Hausdorff dual for this observation. Specifically, we construct a localic groupoid $\mathbb{G}_{KH}$ and identify, for any locale $X$, principal $\mathbb{G}_{KH}$-bundles over $X$ with compact Hausdorff locales internal to $Sh(X)$.

In outline the proof proceeds as follows. Firstly, by showing  (Proposition \ref{prop:Proper_descend}) that proper maps of locales descend along effective descent morphism and that compact Hausdorff locales can be characterised as those locales that have proper diagonals, we see that $X \mapsto \mathbf{KHaus}_X$ is a stack. Then we recall two conditions that are sufficient to show that a stack on the category of locales is geometric; that is, equivalent to $X \mapsto Prin_{\mathbb{G}}(X)$ for some localic groupoid $\mathbb{G}$. The two conditions are similar to the familiar ones appearing in the definition of Artin stacks or Deligne-Munford stacks from algebraic geometry. The first condition is that there is a locale $G_0$ and a canonical object at stage $G_0$ such that every other object is covered by the canonical object via an effective descent morphism. The second condition is that there is an object $G_1$ that, in a sense, classifies isomorphisms between points of the stacks. 

The proof is then completed by checking the two conditions for the case $X  \mapsto \mathbf{KHaus}_X$. Checking the first requires us to recall that every compact Hausdorff locale is the completion of a normal distributive lattice (NDL). We use this to show how any compact Hausdorff locale can be pulled back (via a cover) to a stage at which it is  the completion of the pullback of the generic NDL. (This generic NDL exists because the theory of normal distributive lattices is geometric, and we are able to obtain the compact Hausdorff locale needed as the completion process commutes with pullback functors determined by geometric morphisms.) Constructing the ``isomorphisms classifier'' $G_1$ required for the second condition is reasonably straightforward using the generic compact Hausdorff locale defined via the completion of the generic normal distributive lattice. This is because compact Hausdorff locales are locally compact and so are exponentiable.

We finish by including some comments on how to extend the result to arbitrary morphisms between compact Hausdorff locales, showing how to represent these as $\mathbb{S}$-homotopies using the main result of \cite{HenryTow}. 
\section{Background and preliminary material}

\subsection{Locales}
For background on locales consult Part C of \cite{Elephant}. We will pass through the equivalence $\mathbf{Loc}/X \simeq \mathbf{Loc}_{Sh(X)}$, e.g. C1.6.3 of \cite{Elephant}, without comment. We assume familiarity with the notion of proper and open locale map; a locale $X$ is discrete(compact Hausdorff) if and only if all finite (including nullary) diagonals are open(proper).

Open and proper locale maps are pullback stable, and if they are also surjections then they are effective descent morphisms in the category of locales $\mathbf{Loc}$. A locale map $f: X \rTo Y$ is \emph{of effective descent} if the pullback functor $f^*$ is monadic; equivalently, the canonical map $\mathbf{Loc}/Y \rightarrow [\mathbb{X}_f, \mathbf{Loc}]$ is an equivalence, where $\mathbb{X}_f$ is the localic groupoid determined by the kernel pair of $f$ (we use the notation $[\mathbb{G},\mathbf{Loc}]$ for the category of $\mathbb{G}$-objects for any localic groupoid $\mathbb{G}$). Effective descent morphisms are pullback stable, essentially because monadicity criteria are pullback stable.

Open and proper maps can be isolated using lower ($P^L$) and upper ($P^U$) power locale constructions; a locale map $g: Z \rTo X$ is open if and only if $P^L_X(Z_g)$ has a top element (and is proper if and only if $P^U_X(Z_g)$ has a bottom). See Theorem 4.9 (Theorem 5.10) of \cite{viclocnotpointless}. Here we are following the notation that if $f:X \rTo Y$ is a locale map then we write $X_f$ when considering it as an object of the slice $\mathbf{Loc}/Y$.

\subsection{Principal bundles}
Let $\mathbb{G} = (d_0,d_1:G_1 \pile{\rTo \\ \rTo} G_0,m,s,i)$ be a localic groupoid. Then a \emph{principal $\mathbb{G}$-bundle over $X$} is a $\mathbb{G}$-object $(P_x,a:G_1 \times_{G_0} P \rTo P)$ together with a $\mathbb{G}$-invariant morphism $p:P \rTo X$ such that  $(a,\pi_2):G_1 \times_{G_0} P \rTo P \times_X P$ is an isomorphism and $p$ is an effective descent morphism. 

If $(P_x, a)$ is a principal bundle then the inverse $(a,\pi_2)^{-1}:P \times_X P \rTo G_1 \times_{G_0} P$ is necessarily of the form $(\psi^P,\pi_2)$ where $\psi^P: P \times_X P \rTo G_1$ satisfies the usual cocycle condition. Given an effective descent morphism $p:P \rTo X$ a \emph{cocycle} for $P_p$ with values in $\mathbb{G}$ is an internal functor $\mathbb{P}_p \rTo \mathbb{G}$. So every principal bundle $(P_x,a)$ gives rise to a cocycle with the internal functor determined by $(\psi^P,x)$. In the other direction every cocycle $F :\mathbb{P}_p \rTo \mathbb{G}$ for any effective descent morphism $p:P \rTo X$ gives rise to a principal bundle $(Q_y,b)$; it is unique up to isomorphism with the property that $F$ factors via $(\psi^Q,y)$. The object $Q$ is the quotient on $d_0^*(P_{x}) = G_1 \times_{G_0} P$ found by associating each $(gF_1(p_1,p_2),p_2)$ and $(g,p_1)$. Of course, locales don't have points, but this association can be expressed with a coequalizer diagram; the coequalizer diagram is pullback stable, allowing a $\mathbb{G}$-action to be defined on $Q$.   

If there are two functors $F^1,F^2:\mathbb{P}_p \rTo \mathbb{G}$, giving rise to two principal bundles $Q_1,Q_2$, then internal natural transformations from $F^1$ to $F^2$ are in bijection with $\mathbb{G}$-homomorphisms from $Q_1$ to $Q_2$ over $X$. To see this, first note that because $p$ is an effective descent morphism $[\mathbb{P}_p,\mathbf{Loc}] \simeq \mathbf{Loc}/X$. Therefore morphisms $Q_1 \rTo Q_2$ over $X$ are in bijection with $\mathbb{P}_p$-object homomorphisms $d_0^*P_{F^1_0} \rTo d_0^*P_{F^2_0}$. This specialises and $\mathbb{G}$-homomorphisms from $Q_1$ to $Q_2$ are in bijection with morphisms $d_0^*P_{F^1_0} \rTo d_0^*P_{F^2_0}$ that are both $\mathbb{G}$ and $\mathbb{P}_p$-homomorphisms. Any such morphism must be of the form $(\bar{\alpha},\pi_2)$ as it is over $P$ and moreover $\bar{\alpha}: G_1 \times_{G_0} P \rTo G_1$ must be of the form $(g,p) \mapsto g \alpha(p)$ for some uniquely determined $\alpha:P \rTo G_1$ as $(\bar{\alpha},\pi_2)$ is a $\mathbb{G}$-homomorphism. Then it is easy to see that $(\bar{\alpha},\pi_2)$ is a $\mathbb{P}_p$-homomorphism if and only if $\alpha$ is an internal natural transformation.

For any locale $X$, $Prin_{\mathbb{G}}(X)$ is the category whose objects are principal $\mathbb{G}$-bundles over $X$ and whose morphisms are $\mathbb{G}$-object homomorphisms over $X$. Given that natural transformations between groupoids are always invertible, we know by the correspondence between internal natural transformations and principal bundle maps just outlined that all the morphisms of $Prin_{\mathbb{G}}(X)$ are isomorphisms.  

Finally, $Prin_{\mathbb{G}}(X)$ is pseudo-functorial in $X$; for any locale map $f: Y\rTo X$, pullback along $f$ preserves the property of being a principal $\mathbb{G}$-bundle. 

\subsection{Geometric Theories and Normal Distributive Lattices}
We will also consider geometric theories and their classifying toposes, e.g. Part B4.2 of \cite{Elephant}. For any such theory $T$ and over any base topos $\mathcal{S}$ with a natural numbers object, there is a topos $[T]$ bounded over $\mathcal{S}$ and an equivalence, $T(\mathcal{E}) \simeq \frak{BTop}/\mathcal{S}(\mathcal{E},[T])$ natural in each topos $\mathcal{E}$ bounded over $\mathcal{S}$.

We are going to be interested in the geometric theory of normal distributive lattices $NDL$; these are distributive lattices $D$ with the additional property that for any $a,b \in D$ if $a \vee b = 1$ then there exists $a',b' \in D$ such that $a' \wedge b\ = 0$ and $a' \vee b = 1 = a \vee b'$. We write $\mathbf{NDL}$ for the category of normal distributive lattices with lattice homomorphisms as morphisms. We will not need to be explicit about the underlying lattice theory here as we shall instead quote: 
\begin{proposition}\label{c}
(i) There is an essentially surjective functor $c:\mathbf{NDL} \rTo \mathbf{KHaus}^{op}$ to the opposite of the category of compact Hausdorff locales.

(ii) The functor (i) can be constructed relative to any topos and is stable under geometric morphisms. That is, if $f: \mathcal{F} \rTo \mathcal{E}$ is a geometric morphism and $D$ is a normal distributive lattice in $\mathcal{E}$, then $f^*(c_{\mathcal{E}}(D)) \cong c_{\mathcal{F}}(f^*D)$ where we are using $f^*$ both for the locale pullback functor $\mathbf{Loc}_{\mathcal{E}} \rTo \mathbf{Loc}_{\mathcal{F}}$ and the inverse image of the geometric morphism $f$.  
\end{proposition}
\begin{proof}
Part (i) is covered in \cite{HenryTow}, but is also essentially covered in \cite{spit}. Any compact Hausdorff locale $X$ is isomorphic to $c(\mathcal{O}(X))$, where $\mathcal{O}(X)$ is the frame of opens of $X$ (which is a normal distributive lattice if $X$ is compact Hausdorff). 

Consult \cite{spit} for (ii).    
\end{proof}
\section{Stacks on $\mathbf{Loc}$}
In this section we recall the notion of stack on the category of locales, where our notion of cover is a single morphism consisting of an effective descent morphism. We also give some examples of localic stacks. 
\begin{definition}
Given a pseudo-functor $M:\mathbf{Loc}^{op} \rTo \mathfrak{CAT}$ for any locale map $f:Y \rTo X$ we define the category $Des(M,f)$ of \emph{descent data for $f$ in $\mathcal{C}$} as follows. The objects of $Des(M,f)$ are pairs $(A, \theta_A):M(\pi_1)(A) \rTo M(\pi_2)(A))$, where $A$ is an object of $M(Y)$, $\theta_A$ satisfies the cocycle conditions for $f$ and $\pi_1, \pi_2$ are the two projections $X \times_Y X \pile{\rTo \\ \rTo } X$. Morphisms $(A,\theta_A) \rTo (B, \theta_B)$ consists of maps $\psi: A \rTo B$ compatible with the $\theta$s; that is, $[M(\pi_2)(\psi)] \theta_A = \theta_B [M(\pi_1)(\psi)]$.
\end{definition}
See Definition B1.5.1 of \cite{Elephant} for background and the preamble to Proposition B1.5.5 for the case of a single cover. Of course the morphisms $\theta_A$ are all isomorphisms: their inverses are determined by $M(\tau)(\theta_A)$ where $\tau:  Y \times_X Y \rTo Y \times_X Y$ is the twist isomorphism. There is a canonical functor $M(X)\rTo Des(M,f)$ for any $f:Y \rTo X$: send any object $A$ to $M(f)(A)$. 
\begin{example}
Consider $\mathbb{LOC}: \mathbf{Loc}^{op} \rTo \mathfrak{CAT}$, the pullback pseudo-functor on $\mathbf{Loc}$ which sends any locale $X$ to the slice category $\mathbf{Loc}/X$ and any locale map $f$ to pullback along $f$; e.g. Example B1.2.2(c) of \cite{Elephant}. Then $Des(\mathbb{LOC},f:Y \rTo X)$ is isomorphic to $[\mathbb{Y}_f,\mathbf{Loc}]$. 
\end{example}
\begin{definition}
A pseudo-functor $M:\mathbf{Loc}^{op} \rTo \frak{CAT}$ is a \emph{stack} provided for any effective descent morphism $f:X \rTo Y$, $M(Y)$ is equivalent, via the canonical functor, to $Des(M,f)$.
\end{definition}
\begin{example}
The (contravariant) pullback pseudo-functor $\mathbb{LOC}(X) = \mathbf{Loc}/X$ is a stack, essentially by definition of effective descent morphism. Let $f: Y \rTo X$ be a locale map. The definition of $f^*$ being monadic is that the canonical functor $f^*:\mathbf{Loc}/X \rTo (\mathbf{Loc}/Y)^{\mathbb{T}_f}$ is an equivalence, where $\mathbb{T}_f$ is the monad on $\mathbf{Loc}/Y$ determined by the pullback adjunction $\Sigma_f \dashv f^*$. But the algebras of $\mathbb{T}_f$ are exactly the $\mathbb{Y}_f$-objects for the internal groupoid $\mathbb{Y}_f$ which we have already seen can be identified with $Des(\mathbb{LOC},f)$.   
\end{example}

\begin{example}
Given $\mathbb{G}$ a localic groupoid, then the contravariant pseudo-functor $X \mapsto Prin_{\mathbb{G}}(X)$ is a stack. This is proved for example in \cite{BDesc} (Theorem 4.11) in the case where the notion of cover used is that of an open surjection; but a proof using effective descent morphism instead is exactly the same.
\end{example}

The identification of $\mathbf{Loc}/X$ with $Des(\mathbb{LOC},f)$ for $f$ an effective descent morphism also allows for a short proof that open and proper maps descend:
\begin{proposition} \label{prop:Proper_descend}
Given a localic effective decent morphism $f: Y \rTo X$, if $g: A_a \rTo B_b$ is a locale map over $X$ such that $f^*(g)$ is open (resp. proper) over $Y$ then $g$ is open (resp. proper). 
\end{proposition}
\begin{proof}
By change of base, since the pullback of $f$ along $b:B\rTo X$ is still an effective descent morphism, we can assume that $B=X$ and are left checking that if $P^L_Y(f^*A)$ has a top element then so does $P^L_X(A_a)$. But the pullbacks of the top element $\top: 1 \rTo P^L_Y(f^*A_a)$ along both $\pi_1: Y \times_X Y \rTo Y$ and $\pi_2: Y \times_X Y \rTo Y$ are again top elements (recall that $\mathbf{Loc}$ is order enriched cartesian; that is, pullback presereves the order enrichment). Both pullbacks must therefore be the same by uniqueness of top elements. Therefore $\top$ is a morphism of $Des(\mathbb{LOC},f)$ and corresponds to a top elment for  $P^L_X(A_a)$ via $\mathbf{Loc}/X \simeq Des(\mathbb{LOC},f)$.

The conclusion for proper maps is order dual. 
\end{proof}

\begin{example}
For any locale $X$ let $\mathbf{KHLoc}_X$ be the full subcategory of $\mathbf{Loc}_{Sh(X)}$ consisting of compact Hausdorff locales. As pullback preserves proper maps the assignment $X \mapsto \mathbf{KHLoc}_X$ (and morphisms mapping to pullback) is a pseudo-functor $\mathbf{Loc}^{op} \rTo \frak{CAT}$ which is a sub-pseudo-functor of the previous example in an obvious manner. To prove that it is a stack, we just need to check that proper maps descend along effective descent morphisms; this has been covered in the previous proposition.
\end{example}

\begin{example}
The pseudo-functor $X \mapsto Sh(X)$ is a stack. This follows as in the previous example, but with open maps in place of proper maps; recall that for any locale $X$, $Sh(X) \simeq \mathbf{LH}/X \simeq \mathbf{Dis}_X$. That is, the category of sheaves over $X$ is equivalent to the category of discrete locales internal to the topos $Sh(X)$. 
\end{example}

Any stack $M:\mathbf{Loc}^{op} \rTo \frak{CAT}$ gives rise to a stack of groupoids; consider $X \mapsto M^{\cong}(X)$ where $M^{\cong}(X)$ has the same objects as $M(X)$, but has as morphisms only those morphisms of $M(X)$ that are isomorphisms. Proving this requires the simple verification that all the functors involved in the relevant definitions preserve isomorphisms.  

\section{Geometric stacks}
We now provide some conditions for when a stack on $\mathbf{Loc}$ is a stack of principal bundles for some localic groupoid. 
\begin{definition}
A stack $M:\mathbf{Loc}^{op} \rTo \mathfrak{GPD}$ of groupoids is \emph{geometric} if there exists a localic groupoid $\mathbb{G}$ such that $M(X) \simeq Prin_{\mathbb{G}}(X)$ naturally for every locale $X$.
\end{definition}
\begin{proposition}\label{geom}
Let $M:\mathbf{Loc}^{op} \rTo \mathfrak{GPD}$ be a stack of groupoids. Assume that:

(1) there exists a locale $G_0$ and an object $C$ of $M(G_0)$ such that for any $X$ and any $A \in M(X)$ there exists an effective descent morphism $q: Y \rTo X$ and a morphism $x: Y \rTo G_0$ such that $M(q)(A) \cong M(x)(C)$

(2) there is a locale $G_1$ such that,
\begin{eqnarray*}
\mathbf{Loc}(X,G_1) \cong \{ (f,g,\theta) | f,g:X \rTo G_0, \theta: M(f)(C) \rTo^{\cong} M(g)(C) \}
\end{eqnarray*}
naturally, for any other object $X$ of $\mathbf{Loc}$.

Then $M$ is geometric.
\end{proposition}

\begin{proof}
Certainly $G_0$ and $G_1$ must determine a groupoid $\mathbb{G}$; for example, the identity $G_1 \rTo G_1$ determines via (2) two maps, $d_0:G_1 \rTo G_0$ and $d_1:G_1 \rTo G_0$ and $\theta:M(d_0)(C) \rTo M(d_1)(C)$. The maps $d_0$ and $d_1$ are the domain and codomain maps, and the image of $(d_1,d_0,\theta^{-1})$ under the isomorphism of (2) determines the inverse map $i:G_1 \rTo G_1$. Note that therefore by naturality of (1) the image of any $\psi: X \rTo G_1$ under the isomorphism of (1) is necessarily of the form $(d_0\psi, d_1 \psi, \theta)$. For multiplication, by definition of the pullback $G_1 \times_{G_0} G_1$ (i.e. pull $d_1$ back along $d_0$) and using (2), morphisms $\phi: X \rTo G_1 \times_{G_0} G_1$ are in bijection with 5-tuples $(d_0 \pi_1 \phi , d_1 \pi_1 \phi = d_0 \pi_2 \phi, d_1 \pi_2 \phi, \theta_1, \theta_2)$. Therefore any such $\phi$ gives rise to a morphism $X \rTo G_1$ as the image under the isomorphism of (2) of 
\begin{eqnarray*}
M(d_0 \pi_1 \phi)(C) \cong M(\phi)M(\pi_1)M(d_0)(C) \rTo^{\theta_1} M(\phi)M(\pi_1)M(d_1)(C) \cong \\
M(\phi)M(\pi_2)M(d_0)(C) \rTo^{\theta_2} M(\phi)M(\pi_2)M(d_1)(C) \cong M(d_1 \pi_2 \phi)(C) \text{.}
\end{eqnarray*}
Define $m: G_1 \times_{G_0} G_1 \rTo G_1$ via $\phi = Id_{G_1 \times_{G_0} G_1}$. The unit map $s :G_0 \rTo G_1$ is the inverse image of the triple $(Id_{G_0},Id_{G_0},Id_C)$ under (2). 

Given a principal $\mathbb{G}$-bundle $(Q_x,a)$ over $X$ (say via $q:Q \rTo X$), there is a cocycle $(\psi^Q,x):\mathbb{Q}_q \rTo \mathbb{G}$. By applying (2) at $Q \times_X Q \rTo G_1$ and exploiting the naturality of (2) we obtain an object, $(M(x)(C), \theta_Q)$, of $Des(M,q)$ which, by definition of stack, corresponds uniquely up to isomorphism to an object $A^Q$ of $M(X)$ with the property $M(p)(A^Q) = M(x)(C)$.

Next, let us clarify that if $r:P \rTo Q$ is an effective descent morphism then the cocycle $(\psi^Q (r \times r), xr)$ gives rise to an object $(M(xr)(C), \theta_P)$ of $Des(M,qr)$ which by definition of stack (since the composite $qr$ is an effective descent morphism) corresponds to an object $A^P$ of $M(X)$; but we must have $A^P \cong A^Q$ because $M(qr)(A^P)\cong M(qr)(A^Q)$.

If $n: Q_1 \rTo Q_2 $ is a morphism of principal $\mathbb{G}$-bundles; so, $q_2 n=q_1$ where $q_i:Q_i \rTo X$, then by considering $P_p$, the pullback of $q_2$ along $q_1$ (so $p = q_1 \pi_1 = q_2 \pi_2)$ we obtain an internal natural transformation $\alpha_n: P \rTo G_1 $ from $(\psi^{Q_1}(\pi_1 \times \pi_1),x_1)$ to $(\psi^{Q_2}(\pi_2 \times \pi_2),x_2)$. By (2) this corresponds to a morphism $M(x_1 \pi_1)(C) \rTo M(x_2 \pi_2)(C)$ of $Des(M,p)$ and indeed, by applying (2) in the other direction, we see that such morphisms correspond to internal natural transformations. As $p$ is an effective descent morphism, this morphism of $Des(M,p)$ corresponds to a morphism of $M(X)$, which by the clarification of the previous paragraph must be, up to isomorphism, from $A^{Q_1}$ to $A^{Q_2}$. This determines a functor $Prin_{\mathbb{G}}(X) \rTo M(X)$. 

Since we have commented already how such natural transformations $\alpha$ are in bijection with principal bundle maps from $Q_1$ to $Q_2$ we known that the functor is full and faithful.   

For essential surjectivity say we are given an object $A$ of $M(X)$. Then by (1) there exists $x : Y \rTo G_0$ and $\theta_A: M(\pi_1)[M(x)(C)] \rTo^{\cong} M(\pi_2)[M(x)(C)]$, for some effective descent morphism $p: Y \rTo X$, where $\pi_1,\pi_2:Y \times_X Y \pile{\rTo \\ \rTo } Y$. By (2) there is then a map $\psi^A:Y \times_X Y \rTo G_1$ corresponding to the triple $(x \pi_1, x \pi_2, \theta_A)$. Then by naturality of (2) and the definition of $m$ and $s$ it can be checked that $(\psi^A,x)$ determines a cocycle $\mathbb{Y}_p \rTo \mathbb{G}$ which, by earlier comments, gives rise to a principal bundle $(Q^A_y,a_A)$. Since the cocycle $(\psi^A,x)$ factors through the cocycle determined by $Q^A$ we know that $A \cong A^{Q^A}$ because $M(p)(A) \cong M(x)(C) \cong M(r)M(y)(C)$ where the isomorphisms are in $Des(M,p)$ and where $r$ is the factorisation of $x$ through $y$.
\end{proof}
\begin{remark}
This result is a variant of the well known characterisation of geomtric stack in algebraic geometry (where condition (2) is replaced with the assertion that the diagonal on $M$ is representable). 
\end{remark}

\begin{proposition}\label{main}
The stack $X \mapsto \mathbf{KHLoc}^{\cong}_X$ is geometric
\end{proposition}
\begin{proof}
We check 1. and 2. of Proposition \ref{geom}. 

For 1., let $l:Sh(G_0) \rTo $$ [ NDL ] $ be a localic cover of the classifying topos for $NDL$ which we can assume is an open surjection (e.g. Theorem C5.2.1 of \cite{Elephant}). Let $C$ be the compact Hausdorff locale corresponding to the completion of $l^*G_{NDL}$ where $G_{NDL}$ is the generic normal distributive lattice in $NDL$. So, $C = c_{Sh(G_0)}(l^*G_{NDL})$ using the notation of Proposition \ref{c}. Then for any compact Hausdorff locale $A$ in $Sh(X)$, $\mathcal{O}_X A$ is a normal distributive lattice and so is classified by a geometric morphism $k_A:Sh(X) \rTo $$[NDL]$. The pullback (in the category of bounded toposes) of $l$ along $k_A$ is localic and an open surjection so must be of the form $q:Sh(Y) \rTo Sh(X)$ where the locale map $q$ is an open surjection and so is of effective descent. (Recall open surjections between toposes are pullback stable; C3.1.27 of \cite{Elephant}.) We draw the pullback diagram of geometric morphisms
\begin{diagram}
Sh(Y_A) & \rTo^{x} & Sh(G_0) \\
\dTo^{q} &     & \dTo_l \\
Sh(X) & \rTo^{k_A} & [NDL]\\
\end{diagram}
to clarify which map is $q$ and can then use the following to complete our verification of condition (1) of Proposition \ref{geom}: 
\begin{eqnarray*}
q^*A & \cong & q^* (c_{Sh(X)}\mathcal{O}_{X}(A)) \\
& \cong & c_{Sh(Y)}q^*\mathcal{O}_X(A) \\
& \cong & c_{Sh(Y)}q^* k_A^* G_{NDL} \\
& \cong & c_{Sh(Y)}x^* l^* G_{NDL} \\
& \cong & x^*c_{Sh(G_0)}l^* G_{NDL} \\
& \cong & x^* C \text{.}\\
\end{eqnarray*}
For (2) observe that as the locale $C$ is compact Hausdorff so is $\pi_i^*C$ where $\pi_i: G_0 \times G_0 \rTo G_0$ for $i = 1,2$; compact Hausdorff locales are locally compact and so are exponentiable. Define $G_1 = Iso_{G_0 \times G_0}((\pi_2^*C)^{\pi_1^*C})$; the exponentiation is in the category of locales over $G_0 \times G_0$ and $Iso_{G_0 \times G_0}( \_)$ indicates taking the sublocale of isomorphisms (explicitly, this is constructed as a sublocale of $(\pi_2^*C)^{\pi_1^*C} \times (\pi_1^*C)^{\pi_2^*C}$).
\end{proof}

\section{Sierpi\'{n}ski homotopies}
\begin{definition}
Given a localic groupoid $\mathbb{G}$, if $P_1$ and $P_2$ are two principal $\mathbb{G}$-bundles over $X$ (for some locale $X$) then a \emph{$\mathbb{S}$-homotopy from $P_1$ to $P_2$} consists of a principal $\mathbb{G}$-bundle $Q$ over $\mathbb{S} \times X$ and two isomorphisms, $P_1 \cong (0_{\mathbb{S}} !^X,Id_X)^*Q$ and $P_2 \cong (1_{\mathbb{S}} !^X,Id_X)^*Q$ where $0_{\mathbb{S}},1_{\mathbb{S}}: 1 \pile{\rTo \\ \rTo} \mathbb{S}$ are the bottom and top of $\mathbb{S}$. 
\end{definition}

In good cases principal $\mathbb{G}$-bundles and $\mathbb{S}$-homotopies between them form a category, but this is not always the case. In particular the composition of two Sierpi\'{n}ski homotopies cannot be defined in general, but one can make sense of ``a composition'' of two homotopies by using the $n$-points version of the Sierpi\'{n}ski locale.

Let $\mathbb{S}_n$ be the $n$-point Sierpinski locale; that is, the locale such that $Sh(\mathbb{S}_n) \simeq \mathbf{Set}^{[n]}$, where $[n]$ is the category $ 0 \to 1 \to \dots \to n$. For any topos $\mathcal{T}$ geometric morphisms $\mathbb{S}_n \to \mathcal{T}$ are the same as a series of points of $\mathcal{T}$ and maps between them $p_0 \to \dots \to p_n$. The full subcategory of $\mathfrak{Cat}$ on the $[n]$ is the simplicial category $\Delta$, so the construction above defines a functor $\Delta \to \mathbf{Loc}$, where a functor $f:[n] \to [m]$ induces a geometric morphism $\mathbb{S}_n \to \mathbb{S}_m$ whose inverse image functor $f^* : \mathbf{Set}^{[m]} \to \mathbf{Set}^{[n]} $ is composition with $f$. For any locale $X$, we have a groupoid:

\[  Prin^\Delta_{\mathbb{G}}(X)([n]) := Prin_{\mathbb{G}}(\mathbb{S}_n \times X).  \]

As this is natural in $[n]$, we have define a simplicial groupoid $Prin^\Delta_{\mathbb{G}}(X)$; that is, a functor from $\Delta^{op}$ to the category of groupoids.

\begin{definition}
For a category $\mathcal{C}$, we define a simplicial groupoid $N^{gpd}(\mathcal{C})$, such that $N^{gpd}(\mathcal{C})([n])$ is the groupoid of functors $[n] \to \mathcal{C}$ with natural isomorphisms between them.
\end{definition}

The following is a well known result for $\infty$-categories which immediately reduces to a theorem about ordinary categories (and for which it is possible to give a direct proof without going through the theory of $\infty$-categories):

\begin{proposition} $N^{gpd}$ is a fully faithful functor from the $2$-category of categories, functors and natural isomorphisms, to the $2$-category of simplicial groupoids, pseudo-natrual transformations and pseudo-natural modifications. \end{proposition}

\begin{remark}
  While we will not need it, the essential image of $N^{gpd}$ can be characterized explicitly as those simplicial groupoids $X:\Delta^{op} \to \mathfrak{GPD}$ that satisfy the following three conditions:

  \begin{itemize}
  \item $X$ satisfies the Segal condition; that is, for each $n>1$, is the map

\[ X([n]) \to X([1]) \times_{X([0])} X([1]) \times_{X([0])}  \dots \times_{X([0])} X_1  \]

induced by the maps $[1] \simeq \{i,i+1\} \to [n]$ is an equivalence of groupoids (where the pullbacks are pseudo-pullbacks).

\item $X$ satisfies the Rezk (or completeness) condition. That is, if we define $X^{iso} \subset X([1])$ the full subgroupoid of the  $x \in X([1])$ that ``admit an inverse'' in the sense that there is an element $y \in X([2])$, such that the image of $y$ under $[1] \simeq \{0,1\} \subset [2]$ is (isomorphic to) $x$ and the image of $y$ under $[1] \simeq \{0,2\} \subset [2]$ is (isomorphic to) a degenerate object (an object in the image of $X([0]) \to X([1])$). Then the natural map $X([0]) \to X^{iso}$ induced by $X([0]) \to X([1])$ is an equivalence of groupoids.

  \item The functor $X([n]) \to X([0])^{n+1}$ induced by all the maps $[0] \to [n]$ is fully faithful.

  \end{itemize}

  Indeed the first two conditions essentially correspond to the definition of complete Segal spaces (or Rezk spaces) which (when used in the setting of $\infty$-categories and $\infty$-groupoids) are a way to define $(\infty,1)$-categories. The last condition is important to make sure that the object we get is really a $1$-category and not some special  $(\infty,1)$-category.
\end{remark}

We will abuse language and say that a simplicial groupoid ``is a category'' when it is in the essential image of $N^{gpd}$ and will identify this simplicial groupoid with the corresponding category.

\begin{theorem}
There exists a localic groupoid $\mathbb{G}_{KH}$ such that, naturally in locales $X$, 
\begin{eqnarray*}
\mathbf{KHaus}^{op}_X \simeq Prin^{\Delta}_{\mathbb{G}_{KH}}(X) \text{.}
\end{eqnarray*}
In particular, $Prin^{\Delta}_{\mathbb{G}_{KH}}(X)$ is a category for all locales $X$.

\end{theorem}

To be clear, what we mean here is that $N^{gpd}\left(\mathbf{KHaus}^{op}_X \right) \simeq Prin^{\Delta}_{\mathbb{G}_{KH}}(X)$.

\begin{proof}
  In \cite{HenryTow} it is shown that for any category $\mathcal{C}$, $[\mathcal{C},\mathbf{KHaus}] \simeq \mathbf{KHaus}_{\hat{\mathcal{C}}}$ where $\hat{\mathcal{C}}$ is the topos of presheaves on $\mathcal{C}$. The account in \cite{HenryTow} is constructive and natural in $\mathcal{C}$. So it can be carried out relative to $Sh(X)$ for any locale $X$; applying it to the case $\mathcal{C} = \{ 0 \leq 1 \leq \dots \leq n \}$ we have an equivalence of categories
  \[ [ \{ 0 \leq 1 \leq \dots \leq n \}^{op},\mathbf{KHaus_X}] \simeq \mathbf{KHaus}_{\mathbb{S}_n \times X} \simeq Prin^{\Delta}_{\mathbb{G}_{KH}}(X)([n]),\]
In particular, restricting to invertible arrows on both sides, we obtain exactly that $Prin^{\Delta}_{\mathbb{G}_{KH}}(X)$  identifies with the groupoid nerve of the opposite category of $\mathbf{KHaus_X}$ .    
\end{proof}

To show that this Theorem is the compact Hausdorff dual of the well known fact that there is a classifying localic groupoid for the geometric theory of objects, we re-state this well know result in the following form: 
\begin{theorem}
There exists a localic groupoid $\mathbb{G}_{Dis}$ such that, naturally in locales $X$, 
\begin{eqnarray*}
\mathbf{Dis}_X \simeq Prin^{\Delta}_{\mathbb{G}_{Dis}}(X) \text{.}
\end{eqnarray*}
\end{theorem}
\begin{proof}
To find $\mathbb{G}_{Dis}$ the easiest reference is Corollary 5.2 of \cite{BDesc} from which we can establish $\mathbf{Dis}^{\cong}_X \simeq Prin_{\mathbb{G}_{Dis}}(X)$ by considering the case of the object classifier. We can also reach this conclusion by exploiting the same reasoning as used above to find $\mathbb{G}_{KH}$; rather than $NDL$, use the geometric theory consisting of a single object. Discrete locales are locally compact and so are exponentiable allowing the last step of the proof of Proposition \ref{main} to work. 
As for the correspondence between morphisms in $\mathbf{Dis}$ and $\mathbb{S}$-homotopies, this is clear from Lemma B4.2.3 of \cite{Elephant} given that $Sh(\mathbb{S} \times X) \simeq [ \{ 0 \leq 1 \}, Sh(X) ]$ (take $\mathbb{T}$ to be the object theory $\mathbb{O}$), and this easily generalizes to $Sh(\mathbb{S}_n \times X) \simeq [ \{ 0 \leq 1 \leq \dots \leq n \}, Sh(X) ]$ showing that we have an equivalence of categories

\[ Prin^{\Delta}_{\mathbb{G}_{Dis}}(X)([n]) \simeq [ \{ 0 \leq 1 \leq \dots \leq n \}, \mathbf{Dis}_X ].\]

After restricting to invertible morphisms on both sides we see that $Prin^{\Delta}_{\mathbb{G}_{Dis}}(X)$ is the nerve of $\mathbf{Dis}_X$.
\end{proof}

\end{document}